\documentclass[12pt]{amsart}

\usepackage{DKstyle}

\newcommand{\vertiii}[1]{{\left\vert\kern-0.25ex\left\vert\kern-0.25ex\left\vert #1
    \right\vert\kern-0.25ex\right\vert\kern-0.25ex\right\vert}}
\theoremstyle{plain}

\usepackage{caption}
\usepackage[labelfont=rm]{subcaption}

\usepackage[pagebackref=true, colorlinks]{hyperref}

\hypersetup{pdffitwindow=true,linkcolor=blue,citecolor=blue,urlcolor=blue,menucolor=blue}

\usepackage{comment}

\begin{document}
\title[]{On the mean square of the remainder for the Euclidean lattice point counting problem}
\author{Dubi Kelmer}
\thanks{The author is partially supported by NSF grant DMS-1401747.}
\email{kelmer@bc.edu}
\address{Department of Mathematics, Maloney Hall
Boston College, Chestnut Hill, MA}


\date{\today}%
\dedicatory{}%
\commby{}%

\begin{abstract}
We give mean square bounds for the remainder in the lattice point counting problem, counting the number of lattice points in a large ball in $\R^d$, when averaged over families of shears of the lattice.
\end{abstract}

 \maketitle
\section{Introduction} 
For $\Lambda\subseteq \R^d$ a lattice and $B_T\subset\R^d$,  the ball of radius $T$ centered at the origin, the counting function
$$\cN_{B_T}(\Lambda)=\#\Lambda\cap B_T,$$
grows asymptotically like  $\frac{\vol(B_T)}{\vol(\R^n/\Lambda)}$ and we are interested in studying the asymptotics of the remainder
$$\cR_{B_T}(\Lambda):=\cN_{B_T}(\Lambda)-\frac{\vol(B_T)}{\vol(\R^n/\Lambda)}.$$

A classical result of Landau states that $\cR_{B_T}(\Lambda)=O_\Lambda(T^{d-2+\frac{2}{d+1}})$, where we use notation $A=O(B)$ to mean that $|A|\leq cB$ for some constant $c>0$, and we use subscript to indicate the dependence of the constant on parameters. 
In high dimensions, $d\geq 5$, it is conjectured that $\cR_{B_T}(\Lambda)=O_\Lambda(T^{d-2})$ (which is known for rational lattices \cite{Walfisz1924,Landau1924}, and in general when $d\geq 9$ \cite{BentkusGotze1997}).  In lower dimensions, the conjectured bound is  $\cR_{B_T}(\Lambda)=O_{\Lambda,\epsilon}(T^{2+\epsilon})$ for $d=4$ (which is known rational lattices \cite{Landau1924}); $\cR_{B_T}(\Lambda)=O_{\Lambda,\epsilon}(T^{1+\epsilon})$ for $d=3$ (where the best known bound is $O_{\Lambda,\epsilon}(T^{21/16+\epsilon})$ for a rational lattice  \cite{HeathBrown1999,ChamizoCristobalUbis2009} and $O_{\Lambda,\epsilon}(T^{63/43+\epsilon})$ in general  \cite{Muller1999}); and $\cR_{B_T}(\Lambda)=O(T^{1/2+\epsilon})$ for $d=2$ (where the best known bound is $O_{\Lambda,\epsilon}(T^{131/208+\epsilon})$ \cite{Huxley2003}).

The conjectured bounds described above for $\cR_{B_T}(\Lambda)$ are known to be sharp when $\Lambda=\Z^d$, however, when $d\geq 4$ we expect $\cR_{B_T}(\Lambda)$ to be generically much smaller. For example, a classical result of Schmidt \cite{Schmidt1960} implies that $\cR_{B_{T}}(\Lambda)=O_{\Lambda}({T}^{d/2}\log(T))$ for almost all lattices. Schmidt's result holds for any family of ordered (with respect to inclusion) sets in $\R^d$ with volume going to infinity (not just growing balls). When dealing with growing balls we expect the remainder to be even smaller, and  a conjecture of G{\"o}tze \cite{Gotze1998} states that for a generic lattice $\Lambda$, the bound for the remainder should be of order $O(T^{\frac{d-1}{2}+\epsilon})$.

In this paper we give mean square bounds for the remainder, averaged over a family of shears of a lattice that are compatible with  G{\"o}tze's conjecture. Explicitly, any lattice $\Lambda\subset \R^d$ is of the form $\Lambda=\Z^dg$ with $g\in \GL_d(\R)$, and shears of this lattice are given by $\Z^dug$ with $u\in U_d(\Z)\bk U_d(\R)$ where $U_d<\GL_d$ denotes the subgroup of unipotent upper triangular matrices. Note that replacing $g$ by $\gamma g$ with $\g\in \SL_d(\Z)$ does not change $\Lambda$ but it does change the family of shears. Our main result is a mean square bound over any such family of shears.

\begin{thm}\label{t:main}
For any $g\in \GL_d(\R)$ there is a constant $c=c(g)$ such that for all $T\geq 1$
$$\int_{U_d(\Z)\bk U_d(\R)}|\cR_{B_T}(\Z^dug)|^2du\leq cT^{d-1}\log^2(T),$$
where $du$ is Lebesgue measure on $U_d(\R)\cong \R^{\tfrac{d(d-1)}{2}}$.
\end{thm}

\begin{rem}\label{r:sharp}
This generalizes the result of \cite{Kelmer15}, who proved it for the case of $d=2$.
In that case, it was also shown in \cite{Kelmer15} that this bound is essentially sharp (up to the logarithmic term) by showing that, for $d=2$, there is a sequence $T_k\to\infty$ with $\int_{U_d(\Z)\bk U_d(\R)}\cR_{B_{T_k}}(\Z^du)du\gg T_k^{\frac{d-1}{2}}$.
We can show this also holds for $d=3,5$ (and with some more work also for other small odd $d$'s). While we believe that such a lower bound holds in any dimension we were not able to prove it in this generality; see Remarks \ref{r:Sharp1} and \ref{r:Sharp2}.
\end{rem}

We expect similar mean square bounds to hold when averaging over a much smaller family of shears. As explained above, for $d\leq 3$ it is expected that such bounds to hold point-wise. For higher dimensions some averaging is necessary, and it is an interesting questions to understand what is the smallest family for which such bounds hold. To make this more precise, consider the subgroups $U_{d,l}\leq U_d$ of upper triangular matrices having zero above the diagonal in the first leftmost $l$ columns (so that $U_{d,1}=U_{d}$ is the full group). It seems reasonable that averaging over $U_{d,d-1}$ should be enough. While such a result is currently out of our reach, when $d\geq 4$ we can prove the following refinement of Theorem \ref{t:main} 
\begin{thm}\label{t:refine}
Let $d\geq 4$. For any $g\in \GL_d(\R)$ there is a constant $c=c(g)$ such that for $T\geq 1$
$$\int_{U_{d,l}(\Z)\bk U_{d,l}(\R)}|\cR_{B_T}(\Z^dug)|^2du\leq cT^{d-1}\log^2(T),$$
for any $\ell\leq [\frac d 2]$, where $[\cdot]$ denotes the integer part.
\end{thm}

It is interesting to compare Theorems \ref{t:main} and \ref{t:refine} to other results regarding the mean square of the remainder, when averaged over different families.
One classical example is the result of Kendall \cite{Kendall1948}, who showed that the mean square of the remainder, when averaging over translates of a ball, $B_T(\alpha)$ when varying the center $\alpha\in \R^d/\Z^d$. For this family he showed that the mean square  $\int_{(\R/\Z)^d}|\cR_{B_T(\alpha)}(\Z^d)|^2 d\alpha$,
is bounded by $O(T^{\frac{d-1}{2}})$, and that this bound is sharp.

Another type of mean square averages is over sets of deformation in the full space of lattices. 
For example, Hofmann Iosevich and Weidinger \cite{HofmannIosevichWeidinger2004} considered averaging over the set of lattices of the form $\{\Z^da_\lambda| \lambda\in[1,2]^d\}$ with $a_\lambda$ a diagonal  matrix with $\lambda_1,\ldots,\lambda_d$ on the diagonal. For this family they showed that in dimensions $d=2,3$  
$$\int_{[1,2]^d}|\cR_{B_T}(\Z^d a_\lambda)|^2 d\lambda=O_\epsilon(T^{d-1+\epsilon}),$$ (see also \cite{PetridisToth2002} for results on similar sets of deformations). 
Their result was extended by Holmin  \cite{Holmin13},  who showed that for any compact set $\cC$ in the full space of lattices $\GL_d(\Z)\bk \GL_d(\R)$, again with $d=2,3$, the mean square
$\int_{\cC}|\cR_{B_T}(\Lambda)|^2d\tilde\mu(\Lambda)$ is bounded by $O_\cC(T^{d-1}\log^m(T)),$
for some $m\in \N$, with $\tilde\mu$ induced from Haar measure on $GL_d(\R)$.

It was noted in \cite{HofmannIosevichWeidinger2004}  that using the results of \cite{IosevichSawyerSeeger2002}, bounding the mean square 
$$\frac{1}{T}\int_{T}^{2T}|\cR_{B_T}(\Lambda)|^2dt=O_{\Lambda,\epsilon}(T^{2d-4+\epsilon}),$$ 
one can deduce that $\int_{[1,2]^d}|\cR_{B_T}(\Z^d a_\lambda)|^2 d\lambda$ is bounded by $O_\epsilon(T^{2d-4+\epsilon})$ for any $d\geq 3$. While the results of \cite{IosevichSawyerSeeger2002} are essentially optimal, it is likely that when averaging over the larger set $\{\Z^da_\lambda| \lambda\in[1,2]^d\}$ the correct bound should be of order $O_{\epsilon}(T^{d-1+\epsilon})$, however, to the best of our knowledge such result is currently not known in any dimension $d\geq 4$.
Nevertheless, as a direct consequence of Theorem \ref{t:main} we can get the following result, generalizing the result of \cite{Holmin13} to any dimension.
\begin{cor}\label{c:main}
For any compact set $\cC$ in the space of (unimodular) lattices $X_d=\SL_d(\Z)\bk\SL_d(\R)$ there is a constant $c=c(\cC)$ such that for all $T\geq 1$ 
$$\int_{\cC}|\cR_{B_T}(\Lambda)|^2d\mu(\Lambda)\leq c T^{d-1}\log^2(T)$$
where $\mu$ denotes the probability measure on $X_d$ induced from Haar measure on $\SL_d(\R)$
\end{cor}
\begin{rem}
The assumption that $\cC$ is compact is necessary for this result. First, when $d=2$, the mean square of the remainder over the full space of unimodular lattices diverges so such a result is hopeless. In higher dimensions, the mean square converges and satisfies the bound $\int_{X_d}|\cR_{B_T}(\Lambda)|^2d\mu(\Lambda)=O(T^d)$, which was shown in \cite{Holmin13} to be sharp.
\end{rem}

From the mean square bound over compact sets together with a standard Borel-Cantelli argument, we can deduce the following result, in the spirit of  G{\"o}tze's conjecture
\begin{cor}\label{c:BC}
For any fixed sequence $\{T_k\}_{k\in \N}$ growing exponentially fast (i.e., $T_k\geq q^k$ for some $q>1$) we have that $\cR_{B_{T_k}}(\Lambda)=O_{\Lambda}( T_k^{\tfrac{d-1}{2}}\log^{2}(T_k))$ for $\mu$-a.e. $\Lambda\in X_d$ .
\end{cor}

\subsection*{Acknowledgments}
We thank Andreas Str\"ombergsson for our discussions on this problem and for his valuable comments. We also thank Zeev Rudnick for his comments.

\section{Preliminary estimates}
The proof of Theorem \ref{t:main} relies on an extension of the arguments in \cite{Kelmer15} dealing with the case of $d=2$, combined with an induction on the dimension.
While the arguments are elementary, there are a few estimates that we will need in order to execute them. We start by proving these estimates, after setting up some notations.

\subsection{Notations}
We will denote by $\cN_{T}(g):=\cN_{B_T}(\Z^dg)$ and $\cR_{T}(g):=\cR_{B_T}(\Z^dg)$ and think of these functions as functions on $\GL_d(\R)$ (that are left $\SL_d(\Z)$ and right $\SO(d)$ invariant). We also use the notation $A\ll B$ to mean that there is a constant $c>0$ such that $A\leq c B$ and we use subscripts to indicate the dependence of the constant on parameters. 

\subsection{Oscillatory integrals}
The following oscillatory integrals will play a key role in what follows.
Let $J_\nu(x)$ denote the standard Bessel function with index $\nu$ and define
\begin{equation}\label{e:cJ}
\cJ_{\nu,k}(X)=\int_0^X(X^2-u^2)^{\frac{k-2}{2}}J_{\nu}(2\pi u)u^{\nu+1}du.
\end{equation}
For these integrals we show

\begin{lem}\label{l:cJbound}
For any $\nu>0$ and $k\in \N $ we have 
\begin{equation}\label{e:cJbound}
\cJ_{\nu,k}(X)=O_{\nu,k}( X^{\frac{k+2\nu-1}{2}}).
\end{equation}
\end{lem}
\begin{proof}
We first recall some standard properties of the Bessel function that we will need. First, for large argument $x\geq 1$ we have the asymptotic expansion 
\begin{equation}\label{e:asym}
J_\nu(x)=\sqrt{\frac{2}{\pi x}}\cos(x-\frac{\pi (2\nu+1)}{4})(1+O_\nu(\frac{1}{x})),
\end{equation}
while for $x\to 0$ 
\begin{equation}\label{e:asym2}
\frac{2^\nu}{x^\nu}J_{\nu}(x)\to \frac{1}{\Gamma(\nu+1)}.\end{equation}
Next, from the series expansion of the Bessel function we get the following differential identity 
\begin{equation}\label{e:difJ}
\frac{d}{dx}(x^\nu J_\nu(x))=x^{\nu}J_{\nu-1}(x).
\end{equation}
Using \eqref{e:difJ} and integrating by parts, we get that for any $k\in \N$  
\begin{eqnarray*}
\cJ_{\nu-1,k+2}(X)&=&\int_0^X(X^2-u^2)^{\frac{k}{2}}J_{\nu-1}(2\pi u)u^{\nu}du\\
&=&\frac{1}{2\pi}\int_0^X(X^2-u^2)^{\frac{k}{2}}[J_{\nu}(2\pi u)u^{\nu}]'du\\
&=&\frac{k}{2\pi}\int_0^X(X^2-u^2)^{\frac{k-2}{2}}J_{\nu}(2\pi u)u^{\nu+1}du=\frac{k}{2\pi}\cJ_{\nu,k}(X)
\end{eqnarray*}
Hence, for any $k\in \N$ we have the identity 
\begin{equation}\label{e:cJrec}
\cJ_{\nu,k}(X)=\frac{2\pi}{k}\cJ_{\nu-1,k+2}(X).
\end{equation}
Using \eqref{e:cJrec} it is enough to establish \eqref{e:cJbound} for $k=1,2$, as the general bound will follow by induction.

First, for $k=2$ using \eqref{e:difJ} and the asymptotics $J_\nu(X)\ll X^{-1/2}$ we get
\begin{eqnarray*}
|\cJ_{\nu,2}(X)|&=&|\int_0^XJ_{\nu}(2\pi u)u^{\nu+1}du|\\
&=& |\frac{1}{2\pi}\int_0^X[J_{\nu+1}(2\pi u)u^{\nu+1}]'du|\\
&=& \frac{1}{2\pi}|J_{\nu+1}(2\pi X)|X^{\nu+1} \ll X^{\frac{2\nu+1}{2}}.\\
\end{eqnarray*}
This bound together with \eqref{e:cJrec} implies that $\cJ_{\nu,k}(X)=O_{\nu,k}( X^{\frac{k+2\nu-1}{2}})$ for any $\nu>0$ and even $k\in \N$.

Next, for $k=1$ we have
\begin{eqnarray*}
\cJ_{\nu,1}(X)&=&\int_0^XJ_{\nu}(2\pi u)\frac{u^{\nu+1}}{\sqrt{X^2-r^2}}du.\\
\end{eqnarray*}
Using the asymptotics \eqref{e:asym} for the Bessel function and making the change of variables $u\mapsto X-u$ we get 
\begin{eqnarray*}
\cJ_{\nu,1}(X)&=&\int_0^XJ_{\nu}(2\pi u)\frac{u^{\nu+1}}{\sqrt{X^2-u^2}}du\\
&=&\frac{1}{\pi} \int_0^X\cos(2\pi(u-\frac{(2\nu+1)}{8})\frac{u^{\nu+1/2}}{\sqrt{X^2-u^2}}du+O( \int_0^X\frac{u^{\nu-1/2}}{\sqrt{X^2-u^2}}du)\\
&=&\frac{1}{\pi}\int_0^X\cos(2\pi(\omega- u))\frac{(X-u)^{\nu+1/2}}{\sqrt{u(2X-u)}}du+O(X^{\nu-1/2})\\.
\end{eqnarray*}
with $\omega=X-\tfrac{(2\nu+1)}{8}$. Writing $\cos(2\pi(\omega- u))=\cos(2\pi \omega)\cos(2\pi u)-\sin(2\pi \omega)\sin(2\pi u)$ we see that
\begin{eqnarray*}
\lefteqn{\int_0^X\cos(2\pi(\omega- u))\frac{(X-u)^{\nu+1/2}}{\sqrt{u(2X-u)}}du}\\
&&\leq |\int_0^X \cos (2\pi u) \tfrac{(X-u)^{\nu+1/2}}{\sqrt{u(2X-u)}}du|+|\int_0^X\sin(2\pi u)\tfrac{(X-u)^{\nu+1/2}}
{\sqrt{u(2X-u)}}du|\\
\end{eqnarray*} 
In order to bound the last term, let 
$$a_n=\int_{\frac{n-1}{2}}^{n/2}|\sin(2\pi u)|\tfrac{(X-u)^{\nu+1/2}}
{\sqrt{u(2X-u)}}du,$$  and write the integral
$$\int_0^X\sin(2\pi u)\tfrac{(X-u)^{\nu+1/2}}
{\sqrt{u(2X-u)}}du=\sum_{n=1}^{[2X]} (-1)^n a_n+\int_{[2X]/2}^X\sin(2\pi u)\tfrac{(X-u)^{\nu+1/2} }
{\sqrt{u(2X-u)}}du$$
as an the alternating sum plus a small remainder.
The remainder can be bounded by $\int_{X-1/2}^X\tfrac{(X-u)^{\nu+1/2}}
{\sqrt{u(2X-u)}}du\ll \frac{1}{X}$. Since the function $\frac{(X-u)^{\nu+1/2}}
{\sqrt{u(2X-u)}}$ is positive and monotonously decreasing for $u\in(0,X)$, so is the sequence, $\{a_n\}_{n=1}^{[2X]}$, and
the alternating sum is bounded by its first term
$$a_1\leq \int_0^{1/2}\frac{(X-u)^{\nu+1/2}du}
{\sqrt{u(2X-u)}}\ll X^{\nu}.$$
We thus get that
$|\int_0^X\sin(2\pi u)\frac{(X-u)^{\nu+1/2}du}
{\sqrt{u(2X-u)}}|\ll X^{\nu}$. The same argument also shows that $|\int_0^X\cos(2\pi u)\frac{(X-u)^{\nu+1/2}du}
{\sqrt{u(2X-u)}}|\ll X^{\nu}$, and hence $\cJ_{\nu,1}(X)=O(X^\nu)$. Using this together with \eqref{e:cJrec} implies that $\cJ_{\nu,k}(X)=O_{\nu,k}( X^{\frac{k+2\nu-1}{2}})$ for any $\nu>0$ and odd $k\in \N$, thus concluding the proof.
\end{proof}

Specializing to $\nu=\tfrac{1}{2}$ and recalling that $J_{1/2}(x)=\sqrt{\tfrac{2}{\pi x}}\sin(x)$ we get that the following elementary integral
\begin{equation}\label{e:cI}
\cI_k(X)=\int_0^X (X^2-u^2)^{\frac{k-2}{2}}u\sin(2\pi u)du,
\end{equation}
is a the special case $\cI_k(X)=\pi  \cJ_{1/2,k}(X)$ and in particular we have the bound
\begin{equation}\label{e:cIbound}
\cI_k(X)=O_k( X^{k/2}).
\end{equation}
\begin{rem}
For the proof of Theorem \ref{t:main} we only need to use the oscillatory integral \eqref{e:cI} and the bound \eqref{e:cIbound}, which can be proved by elementary means (without Bessel functions). However, for the proof of Theorem \ref{t:refine} we need the more general bound  \eqref{e:cJbound}.
\end{rem}

 \subsection{Smoothed sums}
With the help of the above oscillatory integral estimates we can give good estimates for smoothed sums.
First for one dimensional sums we show
\begin{lem} \label{l:sum2int}For any $k\in \N$ we have
$$\sum_{|n|<T}(T^2-n^2)^{k/2}=2T^{k+1}\int_0^1(1-t^2)^{k/2}dt+O_k(T^{k/2})$$
\end{lem}
\begin{proof}
Let  $f(t)=(T^2-t^2)^{k/2}$ and consider the integral 
$$S(T)=\int_0^Tf'(t)s(t)dt$$
with $s(t)=1/2-\{t\}$ the odd sawtooth function.

On one hand, using that $f(t)=O(T^{k/2})$ for $[T]\leq t\leq T$, we have
$$S(T)=\sum_{n=1}^{[T]} \int_{0}^1f'(t+n-1)(1/2-t)dt+O(T^{k/2}),$$
and after integrating by parts
\begin{eqnarray*}
S(T)=-\sum_{n=1}^{[T]} f(n)-\frac{f(0)}{2}+\int_0^Tf(t)dt +O(T^{k/2}).
\end{eqnarray*}

On the other hand,  plugging in $f'(t)=k(T^2-t^2)^{\frac{k-2}{2}}t$ and expanding
$$s(t)=\sum_{m=1}^\infty \frac{\sin(2\pi m x)}{\pi m},$$ 
we get that
\begin{eqnarray*}
S(T)&=&k\sum_{m=1}^\infty \frac{1}{m}\int_0^T(T^2-n^2)^{\frac{k-2}{2}}t \sin(2\pi m t)dt\\
&=&k\sum_{m=1}^\infty \frac{\cI_k(mT)}{m^{1+k}}
\end{eqnarray*}
with $\cI_k$ defined in \eqref{e:cI}. Using  \eqref{e:cIbound} we bound  $|\cI_k(mT)|\ll_k (mT)^{k/2}$ hence
\begin{eqnarray*}
|S(T)|\ll\sum_{m=1}^\infty \frac{1}{m^{1+k/2}}T^{k/2}\ll_k T^{k/2}.
\end{eqnarray*}
Finally, combining the two estimates we get that indeed
\begin{eqnarray*}\sum_{|n|<T}(T^2-n^2)^{k/2}&=&2\sum_{n=1}^{[T]} f(n)+f(0)\\&=&2\int_0^T(T^2-t^2)^{k/2}dt+O_k(T^{k/2})\\
&=&2T^{k+1}\int_0^1(1-t^2)^{k/2}dt+O_k(T^{k/2}).
\end{eqnarray*}
\end{proof}
\begin{rem}\label{r:Sharp1}
This estimate was proved in  \cite{Kelmer15} for $k=1$ where it was also shown that the bound on the error is sharp. The same proof also works when $k=2$, however, when $k\geq 3$ the situation is more complicated. We note that for even $k$ the sum $\sum_{|n|<T}(T^2-n^2)^{k/2}$ can be evaluated explicitly in terms of power series to show that the bound is indeed sharp. For example for $k=2$
\begin{eqnarray*}\sum_{|n|<T}(T^2-n^2)
&=&\tfrac{4}{3}T^3-(2\{T\}^2-2\{T\}+\tfrac{1}{3})T+(\tfrac{2}{3}\{T\}^3-\{T\}^2+\tfrac{1}{3}\{T\}).
\end{eqnarray*}
so the bound on the error in Lemma \ref{l:sum2int} is sharp as long as $\{T\}\neq \frac{1}{2}\pm \frac{1}{\sqrt{12}}$.
By a similar calculation, for $k=4$ the bound is sharp as long as $\{T\}\neq 0,1/2$ (that is, if $T$ is not an integer or half integer).
It should be possible to carry this out for any even $k$, however, for $k$ large the combinatorics become unwieldy.
Moreover, for odd $k$ these sums can't be evaluated explicitly, and we could not find a way how to show the bound is sharp for any odd $k>1$.
\end{rem}

In higher dimensions we can get a similar estimate for sufficiently smoothed sums over lattice points in $\R^d$ by using 
Poisson summation and the more general bound \eqref{e:cJbound}. Explicitly, we show the following.
\begin{lem}\label{l:Poisson1}
For any $d\geq 2$, for any lattice $\Lambda\subseteq \R^d$ and any $k\geq d$ we have
$$\mathop{\sum_{v\in \Lambda}}_{\norm{v}\leq T} (T^2-\|v\|^2)^{k/2}=\frac{c_{d,k}}{\vol(\R^d/\Lambda)}T^{k+d}+O_{k,\Lambda}(T^{\frac{d+k-1}{2}}),$$
with $c_{k,d}=\frac{2\pi^{d/2}}{\Gamma(d/2)}\int_0^1 (1-r^2)^{k/2}r^{d-1}dr$.
\end{lem}
\begin{proof}
Let $f:\R^d\to \R$ be given by $f(v)=f_0(\|v\|)$ with $f_0(r)=(T^2-r^2)^{k/2}\chi_T(r)$, where $\chi_T$ is indicator function of $[0,T]$. With these notations we can write 
$$\mathop{\sum_{v\in \Lambda}}_{\norm{v}\leq T} (T^2-\|v\|^2)^{k/2}=\sum_{v\in \Lambda}f(v).$$
Now, let
$$\hat{f}(u)=\int_{\R^d}f(v)e^{2\pi iv\cdot u}dv,$$
denote the Fourier transform of $f$. Since $f(v)=f_0(\|v\|)$ is spherical, then $\hat{f}(u)=F_0(\norm{u})$ is also spherical (see e.g. \cite[Chapter 6.4]{SteinShakarchi03}) with 
$$F_0(\rho)= 2\pi \rho^{1-d/2}\int_0^TJ_{\frac{d}{2}-1}(2\pi\rho r)f_0(r)r^{d/2}dr.$$
In particular, using that $(\frac{2}{x})^{\nu}J_{\nu}(x)\to \frac{1}{\Gamma(\nu+1)}$ as $x\to 0$ we see that
\begin{eqnarray*}
F_0(0)&=&\frac{2\pi^{d/2}}{\Gamma(d/2)}\int_0^\infty f_0(r)r^{d-1}dr\\
&=&\frac{2\pi^{d/2}}{\Gamma(d/2)}\int_0^T (T^2-r^2)^{k/2}r^{d-1}dr
=c_{k,d}T^{k+d}.\\
\end{eqnarray*}
Moreover, for $\rho\geq 1$, using Lemma \ref{l:cJbound} we can bound 
\begin{eqnarray*}
F_0(\rho)&=& 2\pi \rho^{1-d/2}\int_0^TJ_{\frac{d}{2}-1}(2\pi\rho r)(T^2-r^2)^{k/2}r^{d/2}dr\\
&=&2\pi\cJ_{d/2-1,k+2}(\rho T)\rho^{-(d+k)}\ll_{d,k} T^{\frac{k+d-1}{2}}\rho^{-\frac{d+k+1}{2}}.
\end{eqnarray*}
In particular, as long as $k>d-1$ we have that $\hat{f}\in L^1(\R^d)$ and we can apply Poisson summation to get
\begin{eqnarray*}
 \sum_{v\in \Lambda}f(v)&=&\frac{1}{\vol(\R^d/\Lambda)}\sum_{u\in \Lambda'}\hat{f}(u)\\
&=& \frac{c_{k,d}T^{k+d}}{\vol(\R^d/\Lambda)}+O_{k}(T^{\frac{k+d-1}{2}}\sum_{0\neq u\in \Lambda'}
\frac{1}{\|u\|^{\frac{d+k+1}{2}}})
\end{eqnarray*}
with $\Lambda'$ the dual lattice of $\Lambda$ and the implied constant depending on $\vol(\R^d/\Lambda)$ and on the length of the shortest vector in $\Lambda'$. For $k>d-1$ the last sum over $\Lambda'$ converges thus concluding the proof.
\end{proof}

\subsection{Oscillatory sums}
For  $g\in GL_d^+(\R), \lambda\in \R$ and $x\in \R^{d}$ let $H_T(g,\lambda,x)$ denote the following oscillatory sum
\begin{equation}\label{e:HT0}
H_T(g,\lambda, x)=\sum_{\norm{ng}<T}\left(s(\tfrac{\sqrt{T^2-\norm{ng}^2}}{\lambda}- n\cdot x)+s(\tfrac{\sqrt{T^2-\norm{ng}^2}}{\lambda}+n\cdot x)\right),\end{equation}
with $s(x)=\tfrac{1}{2}-\{x\}$ as before. We note for future reference that, by expanding $s(x)$ into its Fourier expansion, we get the following series expansion for this sum.
\begin{equation}\label{e:HT}
H_T(g,\lambda, x)=\frac{2}{\pi}\sum_{m=1}^\infty\frac{1}{m}\sum_{\norm{ng}<T}\sin(\tfrac{2\pi m\sqrt{T^2-\norm{ng}^2}}{\lambda})\cos(2\pi m n\cdot x).\end{equation}
We also consider smoothed versions obtained by integrating $H_T(g,\lambda,x)$ against  suitable kernels.  Explicitly, for $j\geq 1$ we define 
\begin{equation}\label{e:HTj}
H_T^{(j)}( g, \lambda, x)=\int_0^T H_t(g,\lambda, x) (T^2-t^2)^{\tfrac{j-2}{2}} tdt.
\end{equation}
For notational convenience we also denote $H^{(0)}_{T}( g, \lambda, x):=H_{T}( g, \lambda, x)$.
Note that $H^{(j)}_{T}(g,\lambda,x)$ is trivially bounded by $O(T^{d+j})$, however, we expect them to be generically much smaller. As evidence for this we give the following estimates for their average and mean square.

\begin{prop}\label{p:meansquare}
For any $g\in \GL_{d}^+(\R), \lambda>0$ and $T>1$ we have that on average
$$\int_{(\R/\Z)^{d}} H_T(g,\lambda;x)dx=O(1),$$
and in the mean square
$$\int_{(\R/\Z)^{d}} |H_T(g,\lambda;x)|^2dx\ll \cN_T(g)\log^2\cN_T(g).$$
\end{prop}

\begin{proof}
Starting from 
$$H_T(g,\lambda;x)=(1-2\{\tfrac{T}{\lambda}\})+\frac{2}{\pi}\sum_{m=1}^\infty\frac{1}{m}\sum_{0<\norm{ng}<T}\sin(\tfrac{2\pi m\sqrt{T^2-\norm{ng}^2}}{\lambda})\cos(2\pi m n\cdot x),$$
and averaging over $x\in (\R/\Z)^d$ all terms but the first vanish hence 
$$\int_{(\R/\Z)^{d}} H_T(g,\lambda;x)dx=(1-2\{\tfrac{T}{\lambda}\}).$$

Next, fix a large parameter $A\geq 2$ (to be determined later) and separate the sum over $m$ into two ranges, say,
\begin{eqnarray*}
\cJ_1&=&\frac{2}{\pi}\sum_{m=1}^A\frac{1}{m}\sum_{0<\norm{ng}<T}\sin(\tfrac{2\pi m\sqrt{T^2-\norm{ng}^2}}{\lambda})\cos(2\pi mn\cdot x)\\
\cJ_2&=&\frac{2}{\pi}\sum_{m>A}\frac{1}{m}\sum_{0<\norm{ng}<T}\sin(\tfrac{2\pi m\sqrt{T^2-\norm{ng}^2}}{\lambda})\cos(2\pi m n\cdot x)\\
\end{eqnarray*}
 so that 
\begin{eqnarray*}
|H_T(g,\lambda;x))|^2&\leq& 4(|\cJ_1|^2+|\cJ_2|^2+1)
\end{eqnarray*}
We can bound each of the terms as follows: For the first term, using Cauchy-Schwarz on the $m$-sum we get 
\begin{eqnarray*}
|\cJ_1|^2\leq \log(A)\sum_{m\leq A}\frac{1}{m}\bigg|\sum_{0<\norm{ng}<T}\sin(\tfrac{2\pi m\sqrt{T^2-\norm{ng}^2}}{\lambda})\cos(2\pi m n\cdot x)\bigg|^2.
\end{eqnarray*}
For the second term,  exchanging the order of summation and using Cauchy-Schwarz in the $n$-sum we get
\begin{eqnarray*}
|\cJ_2|^2\leq \cN_T(g)\sum_{0<\norm{ng}<T}\bigg|\sum_{m>A}\frac{1}{m} \sin(\tfrac{2\pi m\sqrt{T^2-\norm{ng}^2}}{\lambda})\cos(2\pi mn\cdot x)\bigg|^2 .
\end{eqnarray*}
Note that, in both cases the dependence on $x$ is only in the inner most sum, hence
\begin{eqnarray*}
\lefteqn{\int_{(\R/\Z)^d}|H_T(g,\lambda,x)|^2dx\ll}\\
&& \log(A)\sum_{m\leq A}\frac{1}{m}\int_{(\R/\Z)^d}\bigg|\sum_{0<\norm{ng}<T}\sin(\tfrac{2\pi m \sqrt{T^2-\norm{ng}^2}}{\lambda})\cos(2\pi mn\cdot x)\bigg|^2dx\\
\nonumber &&+ \cN_T(g)\sum_{0<\norm{ng}<T}\int_{(\R/\Z)^d}\bigg|\sum_{m>A}\frac{\sin(\tfrac{2\pi m\sqrt{T^2-\norm{ng}^2}}{\lambda})\cos(2\pi mn\cdot x)}{m} \bigg|^2dx .
\end{eqnarray*}
Using orthogonality, we can evaluate the inner integrals  by 
\begin{eqnarray*}
\int_{(\R/\Z)^d}\bigg|\sum_{\norm{ng}<T}\sin(\tfrac{2\pi m \sqrt{T^2-\norm{ng}^2}}{\lambda})\cos(2\pi mn\cdot x)\bigg|^2dx
&=&\sum_{\norm{ng}<T}\sin^2(\tfrac{2\pi m \sqrt{T^2-\norm{ng}^2}}{\lambda})\\
&&\leq \cN_T(g),
\end{eqnarray*}
and
\begin{eqnarray*}
\int_{(\R/\Z)^d}\bigg|\sum_{m>A}\frac{\sin(\tfrac{2\pi m\sqrt{T^2-\norm{ng}^2}}{\lambda})\cos(2\pi m n\cdot x)}{m} \bigg|^2dx&=&\frac{1}{2^d}\sum_{m>A}\frac{\sin^2(\tfrac{2\pi m\sqrt{T^2-\norm{ng}^2}}{\lambda})}{m^2}\\
&&\ll\frac{1}{A}.
\end{eqnarray*}
Plugging these back we get that
\begin{eqnarray*}
\int_{(\R/\Z)^d}|H_T(g,\lambda,x)|^2dx\ll \log^2(A)\cN_T(g)+\frac{\cN_T(g)^2}{A},
\end{eqnarray*}
and taking $A=\max(2,\cN_T(g))$ concludes the proof.
\end{proof}

For the smoothed  oscillatory sums we can use \eqref{e:cIbound}  to get the following average bounds.
\begin{prop}\label{p:meansquare2}
Let $j\geq 1$. For any $g\in \GL_{d}^+(\R),\; \lambda>0$ and $T>1$ we have on average
$$\int_{(\R/\Z)^{d}} H_T^{(j)}(g,\lambda;x)dx\ll_j  T^{j/2},$$
and in the mean square
$$\int_{(\R/\Z)^{d}} |H_T^{(j)}(g,\lambda;x)|^2dx\ll_j (\lambda T)^j\cN_T(g).$$
\end{prop}
\begin{proof}
Changing the order of summation and integration we get that
\begin{eqnarray*}
H_T^{(j)}( g, \lambda, x)&=&\int_0^T (T^2-t^2)^{\tfrac{j-2}{2}}H_t(g,\lambda, x) tdt\\
&=&\frac{2}{\pi}\int_0^T (T^2-t^2)^{\tfrac{j-2}{2}}\sum_{m=1}^\infty\frac{1}{m}\sum_{\norm{ng}<t}\sin(\tfrac{2\pi m\sqrt{t^2-\norm{ng}^2}}{\lambda})\cos(2\pi m n\cdot x)tdt\\
&=& \frac{2}{\pi}\sum_{m=1}^\infty\frac{1}{m} \sum_{\norm{ng}<T}\left(\int_{\norm{ng}}^T (T^2-t^2)^{\tfrac{j-2}{2}}
\sin(\tfrac{2\pi m\sqrt{t^2-\norm{ng}^2}}{\lambda})tdt\right)\cos(2\pi m n\cdot x).\\
\end{eqnarray*}
Next, in each integral make the change of variables $u=\tfrac{m\sqrt{t^2-\norm{ng}^2}}{\lambda}$ to get that 
$$\int_{\norm{ng}}^T (T^2-t^2)^{\tfrac{j-2}{2}}
\sin(\tfrac{2\pi m\sqrt{t^2-\norm{ng}^2}}{\lambda})tdt
=(\frac{\lambda}{m})^{j}\cI_j(\frac{m}{\lambda}\sqrt{T^2-\norm{ng}^2}),$$
with $\cI_j(X)$ defined in \eqref{e:cI}. Plugging this back we see that
\begin{eqnarray*}
H_T^{(j)}( g, \lambda, x)&=& \frac{2\lambda^j}{\pi}\sum_{m=1}^\infty\frac{1}{m^{1+j}} \sum_{\norm{ng}<T}\cI_j(\frac{m}{\lambda}\sqrt{T^2-\norm{ng}^2})\cos(2\pi m n\cdot x).\\
\end{eqnarray*}

When averaging over $x\in (\R/\Z)^d$ all terms with $n\neq 0$ vanish, and using \eqref{e:cIbound}  we get
\begin{eqnarray*}
\int_{(\R/\Z)^{d}} H_T^{(j)}(g,\lambda;x)dx&=& \frac{2\lambda^j}{\pi}\sum_{m=1}^\infty\frac{1}{m^{1+j}} \cI_j(\frac{m}{\lambda}T)\ll (\lambda T)^{j/2}.\\
\end{eqnarray*}
as claimed.

Next, for the mean square, writing $\frac{1}{m^{1+j}}=\frac{1}{m^{3/4}}\frac{1}{m^{1/4+j}}$ and using Cauchy-Schwartz we get
\begin{eqnarray*}
|H_T^{(j)}( g, \lambda, x)|^2&\ll& \lambda^{2j}\sum_{m=1}^\infty\frac{1}{m^{1/2+2j}}\left|\sum_{\norm{ng}<T}\cI_j(\frac{m}{\lambda}\sqrt{T^2-\norm{ng}^2})\cos(2\pi m n\cdot x)\right|^2.\\
\end{eqnarray*}
Integrating over $x\in (\R/\Z)^d$ we are left with just diagonal terms, giving
\begin{eqnarray*}
\int_{(\R/\Z)^d}|H_T^{(j)}( g, \lambda, x)|^2dx&\ll& \lambda^{2j}\sum_{m=1}^\infty\frac{1}{m^{1/2+2j}}\sum_{\norm{ng}<T}|\cI_j(\frac{m}{\lambda}\sqrt{T^2-\norm{ng}^2})|^2.\\
\end{eqnarray*}
Using \eqref{e:cIbound} we bound
$$|\cI_j(\frac{m}{\lambda}\sqrt{T^2-\norm{ng}^2})|^2\ll (\frac{m}{\lambda})^{j}(T^2-\norm{ng}^2)^{j/2}\ll (\tfrac{mT}{\lambda})^j,$$ 
and summing up the series concludes the proof.

\end{proof}
\section{Inductive formula}
We introduce the following coordinates on $\GL_d(\R)^+$ that are suitable for inducting on dimension. Any $g\in\GL^+_d(\R)$ we can write in the form
$$g=\left(\begin{matrix} g^{(d-1)} & \lambda_d x^{(d-1)}\\ 0 &\lambda_d \end{matrix}\right)k^{(d)},$$ 
with $k^{(d)}\in \SO(d)\;, g^{(d-1)}\in \GL^+_{d-1}(\R), \lambda_d>0$, and $x^{(d-1)}\in \R^{d-1}$. Here $k^{(d)}$ is determined up to left multiplication by 
$\left(\begin{smallmatrix}\tilde{k} &0\\ 0& 1\end{smallmatrix}\right)$ with $\tilde{k}\in \SO(d-1)$ and similarly $g^{(d-1)}$ is determined up to right multiplication by $\tilde{k}\in \SO(d-1)$. We can further decompose 
$$g^{(d-1)}=\left(\begin{matrix} g^{(d-2)} & \lambda_{d-1} x^{(d-2)}\\ 0 &\lambda_{d-1} \end{matrix}\right)k^{(d-1)},$$ 
and so on, denoting by $g^{(l)}\in \GL_{l}^+(\R), x^{(l)}\in\R^l$ the corresponding matrices and vectors for $l=1,\ldots, d-1$. 

\begin{rem}
Comparing these coordinates to the coordinates coming from the $G=UAK$ decomposition with $U$ unitary upper triangular, $A$ diagonal, and $K=\SO(d)$, we can write $g=uak$ as the product of a unitary upper triangular matrix $u\in U$ with $x^{(1)},\ldots x^{(d-1)}$ above the diagonal, a diagonal matrix $a\in A$ with $\lambda_1,\ldots,\lambda_d$ on the diagonal and an orthogonal matrix $k\in K$.
\end{rem}

The following result is a natural generalization of \cite[Lemma 3]{Kelmer15}
\begin{lem}\label{p:reduction}
For $g=\left(\begin{smallmatrix} g^{(d-1)} & \lambda_d x^{(d-1)}\\ 0 &\lambda_d \end{smallmatrix}\right)k$ as above we have 
$$\cN_T(g)=2\lambda_d^{-1}P_T( g^{(d-1)})+H_T( g^{(d-1)},\lambda_d, x^{(d-1)}),$$
where $P_T(g)=\int_0^T \frac{\cN_t(g)t}{\sqrt{T^2-t^2}}dt$, and  $H_T(g,\lambda, x)$ is given in \eqref{e:HT}.
\end{lem}
\begin{proof}
To simplify notation denote by $\lambda=\lambda_d,x=x^{(d-1)}$ and $\tilde{g}=g^{(d-1)}$.
Since the counting function is invariant under rotation we may assume $g=\begin{pmatrix} \tilde g & \lambda x\\ 0 &\lambda \end{pmatrix}$ in which case, a general element,  $v\in \Z^d g$, is of the form 
$$v=(n\tilde g,\lambda(n\cdot x+m)),\; n\in \Z^{d-1},\; m\in \Z.$$
Hence 
\begin{eqnarray*}
\cN_T(g)&=&\#\{(n,m)\in \Z^{d-1}\times \Z : \norm{n\tilde g}^2+\lambda^2 |n\cdot x+m|^2<T^2\}\\
&=& \#\{(n,m)\in \Z^{d-1}\times \Z: |n\cdot x+m|<\tfrac{\sqrt{T^2-\norm{n\tilde g}^2}}{\lambda}\}\\
&=& \sum_{\norm{n\tilde g}<T}\#\{m: |n\cdot x+m|<\tfrac{\sqrt{T^2-\norm{n\tilde g}^2}}{\lambda}\}\\
&=& \sum_{\norm{n\tilde g}<T}[\tfrac{\sqrt{T^2-\norm{n\tilde g}^2}}{\lambda}-n\cdot x]-[-\tfrac{\sqrt{T^2-\norm{n\tilde g}^2}}{\lambda}-n\cdot x]\\
&=&\frac{2}{\lambda} \sum_{\norm{n\tilde g}<T}\sqrt{T^2-\norm{n\tilde g}^2}+ \sum_{\norm{n\tilde g}<T}(1-\{\tfrac{\sqrt{T^2-\norm{n\tilde g}^2}}{\lambda}-n\cdot x\}-\{\tfrac{\sqrt{T^2-\norm{n\tilde g}^2}}{\lambda}+n\cdot x\}).\\
\end{eqnarray*}
For the first sum in the last line, summation by parts gives 
$$\frac{2}{\lambda}\sum_{\norm{n\tilde g}<T}\sqrt{T^2-\norm{n\tilde g}^2}=\frac{2}{\lambda}\int_0^T \frac{\cN_{t}(\tilde g)t}{\sqrt{T^2-t^2}}dt=2\lambda^{-1}P_T(\tilde g),$$
while the second sum is 
$$\sum_{\norm{n\tilde g}<T}\left(s(\tfrac{\sqrt{T^2-\norm{n\tilde g}^2}}{\lambda}-n\cdot x)+s(\tfrac{\sqrt{T^2-\norm{n\tilde g}^2}}{\lambda}+n\cdot x\})\right)=H_T(\tilde g,\lambda,x),$$
thus concluding the proof.
\end{proof}

We can now continue inductively to obtain the following more general formula
\begin{prop}\label{p:inductive}
For any $g\in \GL_d^+(\R)$ and $k< d$ we have constants $c_0,\ldots, c_k$ (depending on the class of $g$ in $U_d(\R)\bk \GL_d(\R)/\SO(d)$) such that
\begin{equation*}
\cN_T(g)=c_k\int_0^T \cN_{t}(g^{(d-k)})(T^2-t^2)^{\tfrac{k-2}{2}}tdt+\sum_{j=0}^{k-1} c_j H_T^{(j)}(g^{(d-j-1)},\lambda_{d-j},x^{(d-j-1)}).
\end{equation*}
\end{prop}
\begin{proof}
For $k=1$ this is just Lemma \ref{p:reduction}, so we may assume that $k>1$. By induction,
$$\cN_T(g)=c_{k-1}\int_0^T \cN_{t}(g^{(d-k+1)})(T^2-t^2)^{\tfrac{k-3}{2}}tdt+\sum_{j=0}^{k-2} c_j H_T^{(j)}(g^{(d-j-1)},\lambda_{d-j},x^{(d-j-1)}).$$
Writing 
$$g^{(d-k+1)}=\begin{pmatrix} g^{(d-k)} & \lambda_{d-k+1} x^{(d-k)}\\ 0 & \lambda_{d-k+1}\end{pmatrix}k^{(d-k+1)},$$
and applying Lemma \ref{p:reduction} to $ \cN_{t}(g^{(d-k+1)})$
we get that
\begin{eqnarray*}
\int_0^T \cN_{t}(g^{(d-k+1)})(T^2-t^2)^{\tfrac{k-3}{2}}tdt&=& 2\lambda_{d-k+1}^{-1}\int_0^TP_t(g^{(d-k)})(T^2-t^2)^{\tfrac{k-3}{2}}tdt\\
&&+ H_T^{(k-1)}(g^{(d-k)}, \lambda_{d-k+1} x^{(d-k)}).
\end{eqnarray*}
Now, for $\tilde{g}=g^{(d-k)}$ write 
$$P_t(\tilde g)=\sum_{\norm{n\tilde g}<t}\sqrt{t^2-\norm{n\tilde g}^2},$$ 
so that 
\begin{eqnarray*}
\int_0^TP_t(\tilde g)(T^2-t^2)^{\tfrac{k-3}{2}}tdt&=& \int_0^T\left(\sum_{\norm{n\tilde g}<t}\sqrt{t^2-\norm{n\tilde g}^2}\right)(T^2-t^2)^{\tfrac{k-3}{2}}tdt\\
&=&\sum_{\norm{n\tilde g}<T}\int_{\norm{n\tilde g}}^T\sqrt{t^2-\norm{n\tilde g}^2}(T^2-t^2)^{\tfrac{k-3}{2}}tdt\\
&=&\sum_{\norm{n\tilde g}<T}\int_{0}^{\sqrt{T^2-\norm{\tilde g}^2}}(T^2-\norm{n\tilde{g}}^2-u^2)^{\tfrac{k-3}{2}}u^2du\\
\end{eqnarray*}
where we made the change of variables $u=\sqrt{t^2-\norm{n\tilde{g}}^2}$. We can evaluate each of the integrals separately as
$$\int_{0}^{\sqrt{T^2-\norm{\tilde g n}^2}}(T^2-\norm{n\tilde{g}}^2-u^2)^{\tfrac{k-3}{2}}u^2du=\alpha_{k} (T^2-\norm{n\tilde{g}}^2)^{k/2},$$
with 
$\alpha_k=\int_{0}^{1}(1-u^2)^{\tfrac{k-3}{2}}u^2du$,
to get that
 \begin{eqnarray*}
\int_0^TP_t(\tilde g)(T^2-t^2)^{\tfrac{k-3}{2}}tdt&=&\alpha_{k} \sum_{\norm{n\tilde g}<T}(T^2-\norm{n\tilde{g}}^2)^{k/2}\\
&=& \alpha_{k}k\int_0^T \cN_{t}(\tilde{g})(T^2-t^2)^{\frac{k-2}{2}}tdt.
\end{eqnarray*}
Plugging this back in we get that indeed 
\begin{eqnarray*}\cN_T(g)&=&c_k\int_0^T\cN_{t}(g^{(d-k)})(T^2-t^2)^{\tfrac{k-2}{2}}tdt\\
&&+\sum_{j=0}^{k-1} c_j H_T^{(j)}(g^{(d-j-1)},\lambda_{d-j},x^{(d-j-1)}).
\end{eqnarray*}
with $c_k=2k\lambda_{d-k+1}^{-1} \alpha_{k}c_{k-1}$.
\end{proof}

\begin{rem}
From the proof we see that the constants $c_k(g)$ depend only on the class of $g$ in $U_d(\R)\bk \GL_d(\R)/\SO(d)$ and are proportional to $(\lambda_{d-k+1}\cdots\lambda_d)^{-1}$. In fact, by comparing the main term in the asymptotics as $T\to \infty$ we get that
$$c_k(g)=\frac{\vol(B_1^{(d)})|\det(g^{(d-k)})|}{\vol(B_1^{(d-k)})|\det(g)| \int_0^1(1-t^2)^{\frac{k-2}{2}}  t^{d-k+1}dt}$$
however, we will not need to use this information. 
\end{rem}

\section{Proof of main results}
We now have all the ingredients needed for the proofs.
\begin{proof}[Proof of Theorem 1]
Apply the inductive formula for $\cN_T(g)$ with $k=d-1$ to get that 
$$\cN_T(g)=c_{d-1}(g)\int_0^T \cN_{t}(g^{(1)})(T^2-t^2)^{\tfrac{d-3}{2}}tdt+\sum_{j=0}^{d-2} c_j(g) H_T^{(j)}(g^{(d-j-1)},\lambda_{d-j},x^{(d-j-1)}),$$
where $c_j=c_j(g)$ are as in Proposition \ref{p:inductive}.
Since 
$\cN_{t}(g^{(1)})=\sum_{|\lambda_1 n|<t}1$, after exchanging the order of summation and integration we can rewrite the  integral in the first term as
\begin{eqnarray*}\int_0^T \cN_{t}(g^{(1)})(T^2-t^2)^{\tfrac{d-3}{2}}tdt&=& \sum_{|\lambda_1 n|<T}\int_{|\lambda n|}^T(T^2-t^2)^{\tfrac{d-3}{2}}tdt\\
&=& \frac{\lambda_1^{d-1}}{d-1}\sum_{| n|<\frac{T}{\lambda_1}}((\tfrac{T}{\lambda_1})^2- n^2)^{\tfrac{d-1}{2}}\\
&=&\left(\tfrac{2\int_0^1(1-t^2)^{\frac{d-1}{2}}dt }{\lambda_1(d-1)}\right)T^d+O((\lambda_1T)^{\frac{d-1}{2}})\\
\end{eqnarray*}
where we used Lemma \ref{l:sum2int} for the last estimate. 
Plugging this back and comparing the main terms in the asymptotic as $T\to \infty$ we see that 
$$\cN_T(g)=\frac{\vol(B_T^{(d)})}{|\det(g)|}+\sum_{j=0}^{d-2} c_j(g) H_T^{(j)}(g^{(d-j-1)},\lambda_{d-j},x^{(d-j-1)}) +O((\lambda_1T)^{\frac{d-1}{2}}),$$
and after subtracting the main term and squaring we get for any $u\in U_d(\Z)\bk U_d(\R)$
$$|\cR_{T}(ug)|^2 \ll (\lambda_1T)^{d-1}+ \sum_{j=1}^{d-1}|c_{j-1}(g)|^2|H_T^{(j-1)}(u_{d-j}g^{(d-j)},\lambda_{d-j+1},x^{(d-j)})|^2,$$
where $u_{l}\in U_l(\Z)\bk U_l(\R)$ is the top left $l\times l$ corner of $u$. 
Integrating over $U_d(\Z)\bk U_d(\R)\cong (\R/\Z)^{\frac{d(d-1)}{2}}$ is the same as integrating over $x^{(j)}\in (\R/\Z)^j$ for $j=1,\ldots, d-1$ and we can use Propositions  \ref{p:meansquare} and \ref{p:meansquare2} to bound
\begin{eqnarray*}
 \lefteqn{\int_{U_{d-1}(\Z)\bk U_{d-1}(\R)} \int_{(\R/\Z)^{d-1}}|H_T(ug^{(d-1)},\lambda_{d},x^{(d-1)})|^2dx^{(d-1)}du}\\
 &&\ll \int_{U_{d-1}(\Z)\bk U_{d-1}(\R)} \cN_{T}(ug^{(d-1)})\log^2(\cN_{T}(ug^{(n-1)}))du\end{eqnarray*}
 and
 \begin{eqnarray*}
\lefteqn{\int_{U_{d-j}(\Z)\bk U_{d-j}(\R)}\int_{(\R/\Z)^{d-j}}|H_T^{(j-1)}(ug^{(d-j)},\lambda_{d-j+1},x^{(d-j)})|^2dx^{(d-j)}du}\\
&& \ll\int_{U_{d-j}(\Z)\bk U_{d-j}(\R)}\cN_{T}(ug^{d-j})du.
\end{eqnarray*}

We thus get that 
\begin{eqnarray*}
\int_{U_d(\Z)\bk U_d(\R)}|\cR_{T}(ug)|^2du&\ll& (\lambda_1T)^{d-1}\\
&&+\int_{U_{d-1}(\Z)\bk U_{d-1}(\R)} \cN_{T}(ug^{(d-1)})\log^2(\cN_{T}(ug^{(n-1)}))du\\
&&+\sum_{j=2}^{d-1}|c_{j-1}(g)|^2\lambda_{d-j+1}^jT^j\int_{U_{d-j}(\Z)\bk U_{d-j}(\R)}\cN_{T}(ug^{d-j})du.\end{eqnarray*}
Finally, applying the trivial estimate $\cN_{T}(g^{(k)})=O_g(T^k)$ and noting that the implied constant here depends only on the class of $g$ in $U_d(\R)\bk \GL_d(\R)/\SO(d)$, we get that 
$$\int_{U_d(\Z)\bk U_d(\R)}|\cR_{T}(ug)|^2du\ll_g T^{d-1}\log^2(T),$$
as claimed.

\end{proof}

\begin{rem}\label{r:Sharp2}
Since the average of the oscillatory sums is small, by averaging the inductive formula we get
$$\int_{U_d(\Z)\bk U_d(\R)}\cN_T(ug)du=\frac{\lambda_1^{d-1}c_{d-1}(g)}{d-1}\sum_{| n|<\frac{T}{\lambda_1}}((\tfrac{T}{\lambda_1})^2- n^2)^{\frac{d-1}{2}}+O(T^{\frac{d-2}{2}}).$$
Using Lemma \ref{l:sum2int} and subtracting the main term we get that
$$\int_{U_d(\Z)\bk U_d(\R)}\cR_T(ug)du\ll O(T^{\frac{d-1}{2}}),$$
and that this bound is sharp if and only if the estimate
in Lemma \ref{l:sum2int}
is also sharp. In particular, this holds when $d=2,3,5$, (see remark \ref{r:Sharp1}).
\end{rem}

\begin{proof}[Proof of Theorem \ref{t:refine}]
Let $l\leq \frac{d}{2}$ and use the inductive formula in Proposition \ref{p:inductive} with $k=d-l$ to get that for any $g\in GL_d(\R)^+$
$$\cN_T(g)=c_{d-l}(g)\int_0^T \cN_{t}(g^{(l)})(T^2-t^2)^{\tfrac{d-l-2}{2}}tdt+\sum_{j=0}^{d-l-1} c_j(g) H_T^{(j)}(g^{(d-i-1)},\lambda_{d-j},x^{(d-j-1)}).$$
Now, let $\Lambda=\Z^lg^{(l)}\subseteq \R^l$, since $k=d-l> l-1$ we can use Lemma \ref{l:Poisson1} to get that
\begin{eqnarray*}
\int_0^T \cN_{t}(g^{(l)})(T^2-t^2)^{\tfrac{d-l-2}{2}}tdt
&=&\frac{1}{d-3}\mathop{\sum_{v\in \Lambda}}_{\|v\|<T}(T^2-\|v\|^2)^{\frac{d-l}{2}}\\
&=& c_{\Lambda}T^d+O_{\Lambda}(T^{\frac{d-1}{2}}).
\end{eqnarray*}
Comparing the main terms in the asymptotics as before and subtracting it from both sides we see that
$$\cR_T(g)=\sum_{j=0}^{d-l-1} c_j(g) H_T^{(j)}(g^{(d-j-1)},\lambda_{d-j},x^{(d-j-1)})+O_g((T^{\frac{d-1}{2}}).$$
Since there is no dependence on the right on $x^{(1)},x^{(2)},\ldots, x^{(l-1)}$, squaring and integrating over $U_{d,l}$ give the same bound as in Theorem \ref{t:main}.
\end{proof}

\begin{proof}[Proof of Corollary \ref{c:main}]
The Haar measure of $\SL_d(\R)$ in the coordinates $g=uak$, is given by $dg=\Delta(a)dudadk$ where $du,da$ and $dk$ denote the Haar measures of $U,A$ and $K$ respectively, and $\Delta(a)$ is the modular function of $UA$ (which is continuous).
Any compact set $\cC$ in $\SL_d(\R)$ is contained in a larger compact set of the form 
$\cC'=\{uak|u\in U(\Z)\bk U(\R),\; a\in \cC_A,\; k\in K\}$ with $\cC_A$ a compact set in $A$.  Corollary \ref{c:main} now immediately follows from Theorem \ref{t:main}.
\end{proof}

\begin{proof}[Proof of Corollary \ref{c:BC}]
Fix a sequence $T_k$ with $T_k\geq q^k$ for some $q>1$. Let $\cC\subseteq X_d$ denote a compact set, and for any $k\in \N$ let 
$$A_k=\{\Lambda\in \cC : |\cR_{T_k}(\Lambda)|\geq T_k^{\frac{d-1}{2}}\log^2(T_k)\}.$$ 
Corollary \ref{c:main} implies that there is a constant $c$ (depending on $\cC$) such that $\mu(A_k)\leq \frac{c}{\log^2(T_k)}$, and hence the series 
$\sum_{k}\mu(A_k)<\infty$ converges. By the Borel-Cantelli Lemma we get that $\mu(\cap_{n\in \N}\cup_{k\geq n} A_k)=0$, hence, $\cR_{B_{T_k}}(\Lambda)=O_{\Lambda}(T_k^{\frac{d-1}{2}}\log^2T_k)$ for $\mu$-a.e $\Lambda\in \cC$. Since this is true for any compact set $\cC$, the same holds for $\mu$-a.e. $\Lambda\in X_d$.
\end{proof}

%
\end{document}